\documentclass[11pt]{amsart}

\usepackage{amsmath,amssymb,amsthm}
\usepackage{amsmath,amssymb,amsthm,amscd}
\usepackage[frame,cmtip,arrow,matrix,line,graph,curve]{xy}
\usepackage{graphpap, color}
\usepackage[mathscr]{eucal}
\usepackage{color}

\def\dfrac{\displaystyle\frac}
\def\dsum{\displaystyle\sum}

\def\dlim{\displaystyle\lim}
\newtheorem{prop}{Proposition}
\newtheorem{theo}[prop]{Theorem}
\newtheorem{lemm}[prop]{Lemma}
\newtheorem{coro}[prop]{Corollary}
\newtheorem{rema}[prop]{Remark}

\newtheorem{defi}[prop]{Definition}

\newcommand{\ra}{\rightarrow}

\newcommand{\pa}{\partial}
\newcommand{\al}{\alpha}
\newcommand{\abc}[1]{\left( #1 \right)}
\newcommand{\abz}[1]{\left[ #1 \right]}
\renewcommand{\leq}{\leqslant}
\renewcommand{\geq}{\geqslant}

\newcommand{\p}{\partial}

\newcommand{\la}{\lambda}

\numberwithin{equation}{section}

\title{The curvature estimates  for convex solutions of some fully nonlinear Hessian type equations }

\begin{document}

\author{Chunhe Li}
\address{School of Mathematical Sciences  \\ University of Electronic Science and Technology of China \\ Chengdu, China} \email{chli@fudan.edu.cn}
\author{Changyu Ren}
\address{School of Mathematical Science\\
Jilin University\\ Changchun\\ China}
\email{rency@jlu.edu.cn}
\author{Zhizhang Wang}
\address{School of Mathematical Science\\ Fudan University \\ Shanghai, China}
\email{zzwang@fudan.edu.cn}
\thanks{Research of the  first author is supported by by  FRFCU  Grant No. ZYGX2016J135 and last author is partially supported  by NSFC Grant No.11301087 and No. 11671069.}
\begin{abstract}
The curvature estimates of quotient curvature equation do not always exist even for convex setting \cite{GRW}. Thus it is natural question to find other type of elliptic equations possessing curvature estimates. In this paper, we  discuss  the existence of curvature estimate for fully nonlinear elliptic equations defined by symmetric polynomials, mainlly, the linear combination of elementary symmetric polynomials.

\end{abstract}
\maketitle


\section{introduction}
The existence of curvature estimates or  $C^2$ estimates  for fully nonlinear elliptic partial differential equations is one of the central  topics in this field. One of the most popular fully nonlinear equation is the $k$-Hessian equations. Let's give some setting and a short review. Suppose $M$ is some $n$ dimensional compact hypersurface in Euclidean spaces $\mathbb{R}^{n+1}$.
We let
$\nu(X), \kappa(X)$ are the outer-normal and principal
curvatures of hypersurface $M\subset R^{n+1}$ at position vector $X$
respectively. The prescribed $k$-Hessian curvature equations are
\begin{eqnarray}\label{hessian}
\sigma_k(\kappa(X))=\psi(X,\nu),
\end{eqnarray}
for $1\leq k\leq n$. When $k=1$, curvature estimate comes from the theory of quasilinear PDE.   If $k=n$, curvature estimate in this case for general $\psi(X, \nu)$ is due to Caffarelli-Nirenberg-Spruck \cite{CNS1}. For $f$ independent of the normal vector, the $C^2$-estimate
was proved by Caffarelli-Nirenberg-Spruck \cite{CNS5, CNSV}  for a general class of fully nonlin-
ear operators, including Hessian type and quotient Hessian type. Ivochkina \cite{I1, I} considered the Dirichlet problem of the above equation on domains in $\mathbb R^n$, $C^2$ estimate was proved there under some extra conditions on the dependence of $\psi$ on $\nu$. In \cite{CW}, $C^2$ estimates for Hessian equations have been studied deeply by Chou-Wang. $C^2$ estimate was also proved for equation of prescribing curvature measures problem in \cite{GLM, GLL}.
 If the function $\psi$ is  convex respect to the normal $\nu$, it is well
known  that the global $C^2$ estimate has been obtained by B. Guan
\cite{B}. Recently, Guan-Ren-Wang \cite{GRW}  obtained global $C^2$ estimates for a closed convex hypersurface  and  partially solved the long standing problem. In the same paper \cite{GRW}, they also proved the estimate for  starshaped $2$-convex hypersurfaces. In \cite{LRW}, Li-Ren-Wang relaxed the convex to $k+1$- convex for any $k$ Hessian equations. In \cite{RW}, Ren-Wang totally solved the case  $k=n-1$, that is the global curvature estimates of $n-1$ convex solutions for $n-1$ Hessian equations. In \cite{CLW}, Chen-Li-Wang  extended these estimates to non Codazzi case in warped product space. In \cite{SX}, Spruck-Xiao extended $2$-convex case in \cite{GRW}  to space forms and give a simple proof for the Euclidean case. We also note the recently important work on the curvature estimates and $C^2$ estimates developed by Guan \cite{B2} and Guan-Spruck-Xiao \cite{GSX}.

The  Hessian equations have a lot of applications in the literature. The  famous  Minkowski problem, namely, prescribed Gauss-Kronecker curvature on the outer normal, has been widely discussed in \cite{N, P1, P3, CY}. Alexandrov also posed the problem of prescribing general Weingarten curvature on outer normals, seeing  \cite{A2, gg}. The prescribing curvature measures problem in convex geometry also has been extensively studied in \cite{A1, P1, GLM, GLL}. In \cite{BK, TW, CNS5}, the prescribing mean curvature problem and Weingarten curvature problem also have been considered and obtained fruitful results.

 The estimates of  Hessian equations with generalized right hand side appear some new geometric applications recently. In \cite{PPZ1}, \cite{PPZ2}, Phong-Picard-Zhang generalized the Fu-Yau equations, which is a complex $2$-Hessian equations depending on gradient term
 on the right hand side. The \cite{PPZ2}, \cite{PPZ3} obtained their $C^2$ estimates using the idea of \cite{GRW}.  In \cite{GL}, Guan-Lu
 consider the curvature estimate for hypersurfaces in high dimensional Riemannian manifolds, which is also a $2$-Hessian equation depending on
 normal.  The estimates in \cite{GRW} is also applied in \cite{X} and \cite{BIS}.

For general type of fully nonlinear equations, curvature estimates
does not always exist. The  quotient type of Hessian equations is defined by 
\begin{eqnarray}\label{ex}
\frac{\sigma_k(\kappa(X))}{\sigma_{l}(\kappa(X))}=f(X, \nu),
\end{eqnarray}
for any positive integers $k,l$. The last section in \cite{GRW}  gives some counterexamples which implies curvature estimates do not hold for these quotient equations even for convex setting.

It is natural to ask that except Hessian equations, which type of elliptic equations can always possess curvature estimates for general right hand side?  It means that we need to consider the following general curvature equations:
\begin{eqnarray}\label{GeralQ}
Q(\kappa(X))=\psi(X,\nu(X)),
\end{eqnarray}
where $Q$ is a symmetric function with respect to $\kappa_1\cdots,\kappa_n$. The above example requires that function $Q$ should exclude the quotient type. Therefore, $Q$ needs some restrictions. In view of proof in \cite{GRW}, the quotient concavity has been extensively used which implies that  "order" is enssential for our function.  Thus, a suitable choice for $Q$ may be the symmetric polynomials. On the other hand, we mainly discuss the convex solutions which means that $Q$ should satisfy elliptic conditions in convex cone,
$$Q>0, \text{ and } Q^{ii}=\frac{\p Q}{\p \kappa_i}>0.$$  Hence, as the first step, the simplest choice for $Q$ to satisfy the above requirements may be the linear combinational  $k$ hessian operators, namely
\begin{eqnarray}\label{QSS}
Q(\kappa)=\sum_{s=1}^k\alpha_s\sigma_s(\kappa),
\end{eqnarray}
where  $\alpha_m$
are non negative constants and $\alpha_k>0$, $k\leq n$.  We also need some quotient concavity assumption which we call Condition (Q) corresponding  $k$ Hessian cases.\\

\noindent {\bf Condition (Q)}: There are $k-1$ polynomials $S_1,\cdots,S_{k-1}$ defined by
$$S_l=\sigma_l+\sum_{s=0}^{l-1}\beta_s^l\sigma_s,$$ where constants $\beta_s^l\geq 0$ such that
$(Q/S_{l})^{1/(k-l)}$ are concave functions for $1\leq l\leq k-1$.  \\

Using Condition (Q), we  obtain the curvature estimates of convex solutions for equation defined by \eqref{QSS} in section 2.
\begin{theo}\label{theo2}
 Suppose $M\subset \mathbb R^{n+1}$ is a closed
convex hypersurface satisfying curvature equation
\begin{eqnarray}\label{EQ}
Q(\kappa(X))=\sum_{s=0}^k\alpha_s\sigma_s(\kappa(X))=\psi(X,\nu(X))
\end{eqnarray}
for some positive function $\psi(X, \nu)\in
C^{2}(\Gamma)$ and nonnegative coefficients $\alpha_0,\cdots,\alpha_k$, where $\Gamma$ is an open neighborhood of unit
normal bundle of $M$ in $\mathbb R^{n+1} \times \mathbb S^n$. Further assume that the polynomial $Q$ satisfies Condition (Q), then
there is a constant $C$ depending only on $n, k$, $\|M\|_{C^1}$,
$\inf \psi $ and $\|\psi\|_{C^2}$, such that
 \begin{equation}\label{5.1}
 \max_{X\in M, i=1,\cdots, n} \kappa_i(X) \le C.\end{equation}
 \end{theo}

The type of linear combinational  $k$ Hessian operators also has been studied and been applied in geometry.  Harvey-Lawson \cite{HL80} have considered the special Lagrangian equations which is one of this type.  Krylov \cite{Kry} also considered such type of equations and
obtained some concavity for the opposite sign of the coefficients. In \cite{GZ}, Guan-Zhang  studyed the curvature estimates for such type of equations with the right hand side not depending on gradient term but with coefficients do depending on the position of the hypersurfaces. The geometry problems in hyperbolic space also raise these type of equations naturally  \cite{EGM}.

According to Theorem \ref{theo2}, the curvature estimates become to find suitable quotient concavity: Condition (Q). At first, let's consider the simplest case in which only two terms in \eqref{EQ} appear,
\begin{align}\label{1.g1}
Q_S^k(\kappa(X)):=\al\sigma_{k-1}(\kappa(X))+\sigma_{k}(\kappa(X))=f(X,\nu(X)),
\quad \forall X\in M,
\end{align}
where  $1\leq k\leq n$ and $\alpha$ is nonnegative. We will call this type of Hessian equations: The Sum Type Equations.
In section 3, we will study quotient concavity of sum type equations. We will prove that
 \begin{theo}\label{theosum}
The sum type equations satisfy the Condition (Q) in their admissible sets cone.
\end{theo}
Here the concept of admissible set will be defined in section 3 and it includes convex cone.

For the general case, we need a general condition to recover Condition (Q). Therefore, in section 4, we introduce the following: Condition (C).\\

\noindent {\bf Condition (C)}:  There is some  $b\in \mathbb{R}^N$ and $N\geq k$ such that for $m=0,1,\cdots, k,$ we have
\begin{eqnarray}\label{C}
\alpha_m'=\frac{(n-k)!\alpha_{k-m}}{(n-k+m)!}=\sigma_m(b).\nonumber
\end{eqnarray}
 In other words, the following real coefficient polynomial only has real roots
$$\sum_{m=0}^N\alpha_m't^m=0.$$ \\

This is something like hyperbolic polynomials \cite{Gar, HL13}. Using Condition (C), in section 4, we will prove that
\begin{theo}\label{theos4}
 If the coefficients of the order $k$ polynomial $Q$ defined by \eqref{QSS}
 satisfies Condition (C), then $Q$ should satisfy Condition (Q) in $\Gamma_k$ cone.
\end{theo}
Here $\Gamma_k$ cone is the Garding's cone. See section 3 for more detail. Using Theorem \ref{theo2} and Theorem \ref{theos4},
we have the following main result for convex solutions in the present paper.
\begin{theo}\label{main}
 Suppose $M\subset \mathbb R^{n+1}$ is a closed
convex hypersurface satisfying curvature equation \eqref{EQ}
for some positive function $\psi(X, \nu)\in
C^{2}(\Gamma)$ and nonnegative coefficients $\alpha_0,\cdots,\alpha_k$ satisfying Condition (C), where $\Gamma$ is an open neighborhood of unit
normal bundle of $M$ in $\mathbb R^{n+1} \times \mathbb S^n$, then
there is a constant $C$ depending only on $n, k$, $\|M\|_{C^1}$,
$\inf \psi $ and $\|\psi\|_{C^2}$, such that
 \begin{equation}\label{5.1}
 \max_{X\in M, i=1,\cdots, n} \kappa_i(X) \le C.\end{equation}
 \end{theo}
With appropriate barrier, we have the following existence result coming from the above main Theorem.
\begin{theo}\label{k-exist}
For equation \eqref{EQ}, assume that the coefficients of $Q$ satisfy Condition (C). Suppose $\psi\in C^2(\mathbb R^{n+1}\times \mathbb S^n)$ is a positive function and suppose there is a constant $r>1$ such that
\begin{eqnarray}\label{5.3}
\psi(X,\frac{X}{|X|})\leq \frac{Q(1,\cdots, 1)}{r^k}\ \ \text{ for }\ \   |X|=r ,
\end{eqnarray}
 and $\psi^{-1/k}(X,\nu)$ is a locally convex in $X \in B_r(0)$  for any fixed $\nu\in \mathbb S^n$, then equation (\ref{EQ}) has a strictly convex $C^{3,\alpha}$ solution inside $\bar B_r$.
\end{theo}

These results also can be extended  to linear combinational Hessian equations defined in some domains in Euclidean space. At last, we give the following Lemmas, which will be needed in our proof.
\par
\begin{lemm} \label{lemm D}
Denote $Sym(n)$ the set of all $n\times n$ symmetric matrices. Let
$F=f(\kappa)$ be a $C^2$ symmetric function defined in some open
subset $\Psi \subset Sym(n)$. At any diagonal matrix $A\in \Psi$
with distinct eigenvalues, let $\ddot{F}(B,B)$ be the second
derivative of $C^2$ symmetric function $F$ in direction $B \in
Sym(n)$, then
\begin{align}
\label{1.9} \ddot{F}(B,B) =  \sum_{j,k=1}^n {\ddot{f}}^{jk}
B_{jj}B_{kk} + 2 \sum_{j < k} \frac{\dot{f}^j -
\dot{f}^k}{{\kappa}_j - {\kappa}_k} B_{jk}^2.
\end{align}
\end{lemm}

The proof of this lemma can be found in \cite{Ball} and \cite{CNS3}.

The paper is organized by 4 sections. In the section 2, we
generalize the curvature estimates of convex solutions for equation
\eqref{EQ} with Condition (Q). In section 3, we study the
admissible set and concavity of sum type equations. Section 4 mainly
studies the quotinet concavity of polynomial defined by \eqref{QSS}
with Condition (C).  The last section  stats some conclusions using 
Condition (C) and previous estimtaes for convex hypersurfaces. We also discuss the admissible
solutions for sum type equations.

\section{The curvature estimates}

In this section, let's consider the global curvature estimates for
linear combination of $k$ Hessian  equations \eqref{EQ}. We mainly
prove Theorem \ref{theo2} in this section.

Since the Lemma proved in \cite{LRW} need some more special property
of the  $\sigma_k$, we have to give some restriction on $Q$.  Thus,
we require $Q$ satisfying equation \eqref{EQ} where $\alpha_m$ are
non negative constants and $\alpha_k>0$ and  satisfying Condition
(Q) defined in the first section. Using the Condition (Q), we have
the following result similar to the Hessian equations. The detail of
the proof can be found in \cite{GLL} and \cite{GRW}.
\begin{lemm}\label{lGuan}
Assume that $k>l$, $W=(w_{ij})$ is a Codazzi tensor which is in
$\Gamma_k$. Denote $\beta=\dfrac{1}{k-l}$.  Then, for
$h=1,\cdots, n$, we have the following inequality
\begin{align}\label{1.4}
&-\dfrac{ Q^{pp,qq}}{ Q}(W)w_{pph}w_{qqh}+\dfrac{ S_l^{pp,qq}}{
S_l}(W)
w_{pph}w_{qqh}\\
\geq& \abc{\dfrac{( Q(W))_h}{ Q(W)}-\dfrac{( S_l(W))_h}{
S_l(W)}} \abc{(\beta-1)\dfrac{( Q(W))_h}{
Q(W)}-(\beta+1)\dfrac{( S_l(W))_h}{ S_l(W)}}.\nonumber
\end{align}
Furthermore, for any $\delta>0$, we have
\begin{align}\label{1.5}
&- Q^{pp,qq}(W)w_{pph}w_{qqh} +(1-\beta+\dfrac{\beta}{\delta})\dfrac{( Q(W))_h^2}{ Q(W)}\\
\geq & Q(W)(\beta+1-\delta\beta) \abz{\dfrac{( S_l(W))_h}{
S_l(W)}}^2 -\dfrac{ Q}{
S_l}(W)S_l^{pp,qq}(W)w_{pph}w_{qqh}.\nonumber
\end{align}
\end{lemm}

Set $u(X)=\langle X ,\nu(X)\rangle$ which is the support function of
$M$. By the assumption of Theorem \ref{theo2} that $M$ is convex
with a $C^1$ bound, $u$ is bounded from below and above by two
positive constants. At every point in the hypersurface $M$, choose a
local coordinate frame $\{ \p/(\p x_1),\cdots,\p/(\p x_{n+1})\}$ in
$\mathbb{R}^n$ such that the first $n$ vectors are the local
coordinates of the hypersurface and the last one is the unit outer
normal vector.  Still denote $\nu$ to be the outer normal vector as
in section 1. We let $h_{ij}$  be the second fundamental form of the
hypersurface $M$.  The following geometric formulas are well known
(e.g., \cite{GLL, GRW}).

\begin{equation}
h_{ij}=\langle\partial_iX,\partial_j\nu\rangle,
\end{equation}
and
\begin{equation}
\begin{array}{rll}
X_{ij}=& -h_{ij}\nu\quad {\rm (Gauss\ formula)}\\
(\nu)_i=&h_{ij}\partial_j\quad {\rm (Weigarten\ equation)}\\
h_{ijk}=& h_{ikj}\quad {\rm (Codazzi\ formula)}\\
R_{ijkl}=&h_{ik}h_{jl}-h_{il}h_{jk}\quad {\rm (Gauss\ equation)},\\
\end{array}
\end{equation}
where $R_{ijkl}$ is the $(4,0)$-Riemannian curvature tensor. We also
have
\begin{equation}
\begin{array}{rll}
h_{ijkl}=& h_{ijlk}+h_{mj}R_{imlk}+h_{im}R_{jmlk}\\
=& h_{klij}+(h_{mj}h_{il}-h_{ml}h_{ij})h_{mk}+(h_{mj}h_{kl}-h_{ml}h_{kj})h_{mi}.\\
\end{array}
\end{equation}

As in \cite{LRW}, we also use the $m$-polynomials.  Here, $m$ should be sufficiently large and determined later. We consider the following test function
\begin{align}
\varphi=\log P_m-mZ\log u, \text{ where } P_m=\sum_j\kappa_j^m.
\end{align}
Here  $Z$ is some undetermined constant. Suppose that
function $\varphi$ achieves its maximum value on $M$ at some
point $X_0$. Rotating the coordinates, we assume that $(h_{ij})$ is
diagonal matrix at $X_0$, and $\kappa_1\geq \kappa_2\cdots\geq
\kappa_n$.

\par
Differentiating our test function twice and using Lemma \ref{lemm D}, at $X_0$, we have
\begin{equation}\label{0.2}
\dfrac{\dsum_j\kappa_j^{m-1}h_{jji}}{P_m}-Z\frac{\langle X,\p_i\rangle}{u}h_{ii}=0,
\end{equation}
and we have
\begin{align}\label{0.3}
0\geq &\dfrac{1}{P_m}[\dsum_j\kappa_j^{m-1}h_{jjii}+(m-1)\dsum_j\kappa_j^{m-2}h_{jji}^2
+\dsum_{p\neq q}\dfrac{\kappa_p^{m-1}-\kappa_q^{m-1}}{\kappa_p-\kappa_q}h_{pqi}^2] \\
&- \frac{Z\sum_lh_{iil}\langle \p_l,X \rangle}{u}-\frac{
Zh_{ii}}{u}+Zh_{ii}^2+Z\frac{h_{ii}^2\langle
X,\p_i\rangle^2}{u^2}\nonumber
\end{align}
\par
At $X_0$, differentiating equation \eqref{EQ} twice, we have
\begin{align}\label{e2.14}
Q^{ii}h_{iik}=d_X\psi(\p_k)+h_{kk}d_{\nu}\psi(\p_k),
\end{align}
and we also have
\begin{eqnarray}\label{e2.15}
Q^{ii}h_{iikk}+Q^{pq,rs}h_{pqk}h_{rsk}\geq
-C-Ch_{11}^2+\sum_lh_{lkk}d_{\nu}\psi(\p_l),
\end{eqnarray}
where $C$ is some uniform constant depending on $C^0$ and $C^1$ setting of the hyperurface.  Using (\ref{0.2}) and (\ref{e2.14}), we have
$$
\dfrac{1}{P_m}\dsum_l\dsum_s\kappa_l^{m-1}d_{\nu}\psi(\p_s)h_{sll}-\frac{Z}{u}\sum_kQ^{ii}h_{iik}\langle  \p_k, X\rangle\\
=-\frac{Z}{u}\sum_kd_X\psi(\p_k)\langle X,\p_k\rangle.\nonumber
$$
On the other hand, using Lemma \ref{lemm D}, we have
\begin{align*}
Q^{pq,rs}h_{pql}h_{rsl}=&Q^{pp,qq}h_{ppl}h_{qql}- \sum_{p\neq
q}Q^{pp,qq}h_{pql}^2.
\end{align*}
Then, contacting $Q^{ii}$ in both side of (\ref{0.3}), and using
(\ref{e2.14}), (\ref{e2.15}), we have
\begin{align}
0\geq&\dfrac{1}{P_m}[\sum_l\kappa_l^{m-1}(-C-C(K)h_{11}^2
+KQ_l^2- Q^{pp,qq}h_{ppl}h_{qql}+\sum_{p\neq q}Q^{pp,qq}h_{pql}^2)\nonumber\\
&+(m-1) Q^{ii}\dsum_j\kappa_j^{m-2}h_{jji}^2
+ Q^{ii}\dsum_{p\neq q}\dfrac{\kappa_p^{m-1}-\kappa_q^{m-1}}{\kappa_p-\kappa_q}h_{pqi}^2]\nonumber\\
&-\dfrac{m Q^{ii}}{P_m^2}(\dsum_j \kappa_j^{m-1}h_{jji})^2 +(Z-1)
Q^{ii}h_{ii}^2 +ZQ^{ii}\frac{h_{ii}^2\langle
X,\p_i\rangle^2}{u^2}-C(Z,K)\kappa_1.\nonumber
\end{align}
\par
Let's deal with the third order derivatives.
Denote
\par
$A_i=\dfrac{\kappa_i^{m-1}}{P_m}(KQ_i^2-\dsum_{p,q}
Q^{pp,qq}h_{ppi}h_{qqi})$, ~~
$B_i=\dfrac{2}{P_m}\dsum_j \kappa_j^{m-1}Q^{jj,ii}h_{jji}^2$,
\par
$C_i=\dfrac{m-1}{P_m} Q^{ii}\dsum_j\kappa_j^{m-2}h_{jji}^2$,~~
$D_i=\dfrac{2 Q^{jj}}{P_m}\dsum_{j\neq
i}\dfrac{\kappa_j^{m-1}-\kappa_i^{m-1}}{\kappa_j-\kappa_i}h_{jji}^2$,
\par
$E_i=\dfrac{m Q^{ii}}{P_m^2}(\dsum_j \kappa_j^{m-1}h_{jji})^2$.
\par
We divide two cases to deal with the third order deriavatives, $i\neq 1$ and $i=1$.
\begin{lemm}\label{lemma8}
For any $i\neq 1$, we have
$$A_i+B_i+C_i+D_i-E_i\geq 0,$$ for sufficiently large $m$.
\end{lemm}
\begin{proof}
 At first, by Lemma \ref{lGuan}, for sufficiently large $K$, we have
\begin{equation}\label{0.8}
KQ_l^2- Q^{pp,qq}h_{ppl}h_{qql}\geq
Q(1+\dfrac{\beta}{2})[\dfrac{( S_1)_l}{ S_1}]^2\geq 0.
\end{equation}
Hence, $A_i\geq 0$.
\par
Then, we also have
\begin{align}\label{0.9}
&P_m^2[B_i+C_i+D_i-E_i]\\
=&\sum_{j\neq i}P_m[2\kappa_j^{m-1} Q^{jj,ii}+(m-1)\kappa_j^{m-2}
Q^{ii}
+2 Q^{jj}\sum_{l=0}^{m-2}\kappa_i^{m-2-l}\kappa_j^l]h_{jji}^2\nonumber\\
&+P_m(m-1) Q^{ii}\kappa_i^{m-2}h_{iii}^2\nonumber\\
&-m Q^{ii}(\sum_{j\neq
i}\kappa_j^{2m-2}h_{jji}^2+\kappa_i^{2m-2}h_{iii}^2+\dsum_{p\neq
q}\kappa_p^{m-1}\kappa_q^{m-1}h_{ppi}h_{qqi})\nonumber.
\end{align}
Note that
\begin{align*}
\kappa_j Q^{jj,ii}+ Q^{jj}= \kappa_iQ^{ii,jj}+ Q^{ii} \geq  Q^{ii}.
\end{align*}
For any index $j\neq i$, using the above inequality, we have,

\begin{align}\label{0.11}
&P_m[2\kappa_j^{m-1} Q^{jj,ii}+(m-1)\kappa_j^{m-2} Q^{ii}
+2 Q^{jj}\sum_{l=0}^{m-2}\kappa_i^{m-2-l}\kappa_j^l]h_{jji}^2\\
&-mQ^{ii}\kappa_j^{2m-2}u_{jji}^2\nonumber\\
\geq&(m+1)(P_m-\kappa_j^m) Q^{ii}\kappa_j^{m-2}h_{jji}^2+2P_m
Q^{jj}(\sum_{l=0}^{m-3}\kappa_i^{m-2-l}\kappa_j^l)h_{jji}^2\nonumber.
\end{align}

Using Cauchy-Schwarz inequality, we have
\begin{align}\label{0.12}
2\sum_{j\neq i}\sum_{p\neq i,j}\kappa_j^{m-2}\kappa_p^{m}h_{jji}^2\geq 2\dsum_{p\neq q;p,q\neq
i}\kappa_p^{m-1}\kappa_q^{m-1}h_{ppi}h_{qqi}.
\end{align}
Hence, using \eqref{0.11} and \eqref{0.12} in \eqref{0.9}, we obtain
\begin{align}\label{0.13}
&P_m^2(B_i+C_i+D_i-E_i)\\
\geq&\sum_{j\neq i}[(m+1)\kappa_i^m\kappa_j^{m-2} Q^{ii}+2P_m
Q^{jj}\sum_{l=0}^{m-3}\kappa_i^{m-2-l}\kappa_j^{l}]h_{jji}^2\nonumber\\
&+((m-1)(P_m-\kappa_i^m) -\kappa_i^m)\kappa_i^{m-2}
Q^{ii}h_{iii}^2-2m Q^{ii}\kappa_i^{m-1}h_{iii}\sum_{j\neq
i}\kappa_j^{m-1}h_{jji}\nonumber.
\end{align}
Obviously, $Q^{jj}>Q^{ii}$ for $\kappa_j<\kappa_i$ and
$\kappa_jQ^{jj}>\kappa_iQ^{ii}$ for $\kappa_j>\kappa_i$. Thus, for
$m\geq 6$, and only use $l=m-3,m-4$ we get
\begin{align*}
2P_m Q^{jj}\sum_{l=0}^{m-3}\kappa_i^{m-2-l}\kappa_j^{l}=&2P_m
Q^{jj}(\kappa_i\kappa_j^{m-3}+\kappa_i^{2}\kappa_j^{m-4})+2P_m
Q^{jj}\sum_{l=0}^{m-5}\kappa_i^{m-2-l}\kappa_j^{l}\\
\geq &4\kappa_i^m\kappa_j^{m-2}Q^{ii}+2P_m
Q^{jj}\sum_{l=0}^{m-5}\kappa_i^{m-2-l}\kappa_j^{l}.
\end{align*}
Then,  \eqref{0.13} becomes
\begin{align}\label{0.17}
&P_m^2(B_i+C_i+D_i-E_i)\\
\geq&\sum_{j\neq i}(m+5)\kappa_i^m\kappa_j^{m-2} Q^{ii}h_{jji}^2
+((m-1)(P_m-\kappa_i^m) -\kappa_i^m)\kappa_i^{m-2}
Q^{ii}u_{iii}^2+ \nonumber\\
&-2m Q^{ii}\kappa_i^{m-1}h_{iii}\sum_{j\neq
i}\kappa_j^{m-1}h_{jji}+2P_m
\sum_{j\neq i}Q^{jj}\sum_{l=0}^{m-5}\kappa_i^{m-2-l}\kappa_j^{l}h_{jji}^2\nonumber\\
\geq&(m+5)\kappa_i^m\kappa_1^{m-2} Q^{ii}h_{11i}^2
+((m-1)\kappa_1^m-\kappa_i^m)\kappa_i^{m-2}
Q^{ii}u_{iii}^2 \nonumber\\
&-2m Q^{ii}\kappa_i^{m-1}h_{iii}\kappa_1^{m-1}h_{11i}\nonumber\\
\geq& 0.\nonumber
\end{align}
Here, we have used, for $m\geq 6$, $$(m+5)(m-1)\geq m^2\quad
and\quad  (m+5)(m-2)\geq m^2.$$
\end{proof}

The left case is $i=1$.  Let's continue to prove the  following Lemma which is modified from \cite{GRW}.

\begin{lemm}\label{lemma2}
Suppose that function $Q$ satisfying Condition (Q). For
$\mu=1,\cdots, k-1$, suppose that there is  some positive constant
$\delta\leq 1 $ satisfying  $\kappa_{\mu}/\kappa_1\geq \delta$ and
$\alpha_1=\cdots=\alpha_{\mu}=0$ in \eqref{EQ}. Then there exits
another sufficient small positive constant $\delta'$ depending on
$\delta$,  such that, if  $\kappa_{\mu+1}/\kappa_1\leq \delta'$,  we
have
$$A_1+B_1+C_1+D_1-E_1\geq 0.$$
\end{lemm}
\begin{proof}
At first, note that $Q^{jj}\geq Q^{11}$ for all $j>1$, just like
\eqref{0.17} in Lemma \ref{lemma8}, for $m\geq 6$, we have
\begin{align}\label{3.19}
&P_m^2(B_1+C_1+D_1-E_1)\\
\geq&-\kappa_1^{2m-2} Q^{11}h_{111}^2+2P_m
\sum_{j\neq 1}Q^{jj}\sum_{l=0}^{m-5}\kappa_1^{m-2-l}\kappa_j^{l}h_{jj1}^2\nonumber\\
\geq&-\kappa_1^{2m-2} Q^{11}h_{111}^2
+2P_m\kappa_1^{m-2}\dsum_{j\neq 1} Q^{jj}h_{jj1}^2\nonumber.
\end{align}
Using Condition (Q) and Lemma \ref{lGuan}, we have
\begin{align}\label{3.20}
A_1\geq &\frac{\kappa_1^{m-1} Q}{P_m
S_{\mu}^2}[(1+\frac{\beta}{2})\sum_{a}(
S_{\mu}^{aa}h_{aa1})^2+\frac{\beta}{2}\sum_{a\neq b} S_{\mu}^{aa}
S_{\mu}^{bb}h_{aa1}h_{bb1}\\
&+\sum_{a\neq b}( S_{\mu}^{aa}
S_{\mu}^{bb}- S_{\mu} S_{\mu}^{aa,bb})h_{aa1}h_{bb1}].\nonumber
\end{align}
For $\mu=1$,  notice that $ S_1^{aa}=1$ and $ S_1^{aa,bb}=0$. Then,
we have
\begin{align}
(1+\frac{\beta}{2})\sum_{a,b} h_{aa1}h_{bb1}
\geq&(1+\frac{\beta}{4})h_{111}^2-C_{\beta}\sum_{a\neq
1}h_{aa1}^2.
\end{align}
Then, by equation \eqref{EQ}, we obtain
\begin{align}\label{3.22}
P_m^2A_1
\geq&\frac{P_m\kappa_1^{m-2} Q^{11}}{(1+\sum_{j\neq
1}\kappa_j/\kappa_1+\beta_0^1/\kappa_1)^2}(1+\frac{\beta}{4})h^2_{111}
-\frac{C_{\beta} P_m\kappa_1^{m-1}}{ S_1^2}\sum_{a\neq 1}h_{aa1}^2\\
\geq&P_m\kappa_1^{m-2} Q^{11}h_{111}^2 -\frac{C_{\beta}
P_m\kappa_1^{m-1}}{ S_1^2}\sum_{a\neq 1}h_{aa1}^2.\nonumber
\end{align}
The last two inequalities comes from
$$ Q\geq \kappa_1 Q^{11},\text{ and }
\frac{\kappa_j}{\kappa_1}\leq \delta' ; 1+\frac{\beta}{4}\geq
(1+n\delta')^2.
$$
For $\mu\geq 2$ and $a\neq b$, we have
\begin{align}\label{3.24}
&S_{\mu}^{aa} S_{\mu}^{bb}- S_{\mu} S_{\mu}^{aa,bb}=S_{\mu-1}^2(\kappa|ab)- S_{\mu}(\kappa|ab) S_{\mu-2}(\kappa|ab).
\end{align}
It is obvious that we have, for any
$a,b\leq \mu$,
\begin{equation}\label{3.25}
 S_\mu^{aa}\geq  \frac{\kappa_1\cdots\kappa_{\mu}}{\kappa_a}; \ \ S^{bb}_{\mu}\geq \frac{\kappa_1\cdots\kappa_{\mu}}{\kappa_b}.
 \end{equation}
 For $a,b\leq \mu$, if $\kappa_1$ is sufficient large, we have
 \begin{align}\label{3.26}
  S_{\mu-1}(\kappa|ab)&\leq
 C(\frac{\kappa_1\cdots\kappa_{\mu}}{\kappa_a\kappa_b}+\frac{\kappa_1\cdots\kappa_{\mu+1}}{\kappa_a\kappa_b})\leq C\frac{1+\kappa_{\mu+1}}{\kappa_{b}}S_{\mu}^{aa},\\
  S_{\mu}(\la|ab)&\leq C(\frac{\kappa_1\cdots\kappa_{\mu}}{\kappa_a\kappa_b}+\frac{\kappa_1\cdots\kappa_{\mu+1}}{\kappa_a\kappa_b}+\frac{\kappa_1\cdots\kappa_{\mu+2}}{\kappa_a\kappa_b})\nonumber\\
 &\leq C\frac{1+\kappa_{\mu+1}+\kappa_{\mu+1}\kappa_{\mu+2}}{\kappa_b}S^{aa}_{\mu}, \nonumber\\
  S_{\mu-2}(\kappa|ab)&\leq C\frac{\kappa_1\cdots\kappa_{\mu}}{\kappa_a\kappa_b}\leq C\frac{1}{\kappa_b}S^{aa}_{\mu}.\nonumber
 \end{align}
Then, by (\ref{3.26}), we have, for any undetermined positive
constant $\epsilon$,
\begin{align}\label{3.28}
&\sum_{a\neq b; a,b\leq \mu}( S_{\mu}^{aa} S_{\mu}^{bb}- S_{\mu} S_{\mu}^{aa,bb})h_{aa1}h_{bb1}\\
\geq &-\sum_{a\neq b; a,b\leq \mu}( S_{\mu-1}^2(\kappa|ab)+S_{\mu}(\kappa|ab) S_{\mu-2}(\kappa|ab))h_{aa1}^2\nonumber\\
\geq&-\frac{C_2}{\delta^2}\frac{(1+\delta'\kappa_1)^2+1+\delta'\kappa_1+(\delta')^2\kappa_1}{\kappa^2_1}\sum_{a\leq\mu}(
S^{aa}_{\mu}h_{aa1})^2 \geq  -\epsilon\sum_{a\leq\mu}(
S_{\mu}^{aa}h_{aa1})^2\nonumber.
\end{align}
Here, we choose a sufficiently small $\delta'$, such that
\begin{align}\label{3.29}
\delta'\leq& \frac{\epsilon\delta^2}{5C_2},\ \  \frac{1}{\kappa_1}\leq \delta \sqrt{\frac{\epsilon}{5C_2}}.
\end{align}
For $a\leq\mu$ and $b>\mu$, we have
$$S_{\mu}^{aa}\geq \frac{\kappa_1\cdots\kappa_{\mu}}{\kappa_{a}}; \ \ S^{bb}_{\mu}\geq \kappa_1\cdots\kappa_{\mu-1}.$$
Then,
 for $a\leq \mu, b>\mu $, if $\kappa_1$ is sufficient large, we have
 \begin{align}\label{3.30}
  S_{\mu-1}(\kappa|ab)&\leq
 C\frac{\kappa_1\cdots\kappa_{\mu}}{\kappa_a}\leq CS_{\mu}^{aa} \text{or } CS^{bb}_{\mu},\\
  S_{\mu}(\la|ab)&\leq C(\frac{\kappa_1\cdots\kappa_{\mu}}{\kappa_a}+\frac{\kappa_1\cdots\kappa_{\mu+1}}{\kappa_a})
 \leq C(1+\kappa_{\mu+1})S^{aa}_{\mu}, \nonumber\\
  S_{\mu-2}(\kappa|ab)&\leq C\kappa_1\cdots\kappa_{\mu-2}\leq CS_{\mu}^{bb}.\nonumber
 \end{align}
By (\ref{3.30}), we also have
\begin{align}\label{3.31}
&2\sum_{a\leq \mu; b> \mu}( S_{\mu}^{aa} S_{\mu}^{bb}- S_{\mu} S_{\mu}^{aa,bb})h_{aa1}h_{bb1}\\
\geq&-2\sum_{a\leq \mu; b>\mu} (S^2_{\mu-1}(\kappa| ab)+S_{\mu}(\kappa|ab)S_{\mu-2}(\kappa|ab))|h_{aa1}h_{bb1}|\nonumber\\
\geq&-\epsilon\sum_{a\leq \mu; b>\mu}(
S_{\mu}^{aa}h_{aa1})^2-\frac{C}{\epsilon}\sum_{a\leq \mu; b>\mu}(
S_{\mu}^{bb}h_{bb1})^2\nonumber.
\end{align}
For $a, b>\mu$, we have
$$S_{\mu}^{aa}\geq \kappa_1\cdots\kappa_{\mu-1}; \ \ S^{bb}_{\mu}\geq \kappa_1\cdots\kappa_{\mu-1}.$$
Then,
 for $a, b>\mu $, if $\kappa_1$ is sufficient large, we have
 \begin{align}\label{3.32}
  S_{\mu-1}(\kappa|ab)&\leq
 C\kappa_1\cdots\kappa_{\mu-1},\ \  S_{\mu}(\la|ab)\leq C\kappa_1\cdots\kappa_{\mu}\\
  S_{\mu-2}(\kappa|ab)&\leq C\kappa_1\cdots\kappa_{\mu-2}.\nonumber
 \end{align}
By (\ref{3.32}), we have
\begin{align}\label{3.33}
\sum_{a\neq b; a,b> \mu}( S_{\mu}^{aa} S_{\mu}^{bb}- S_{\mu}
S_{\mu}^{aa,bb})h_{aa1}h_{bb1} \geq&-C\sum_{a\neq b; a,b>\mu}(
S_{\mu}^{aa}h_{aa1})^2.
\end{align}
Hence, combing (\ref{3.20}), (\ref{3.28}), (\ref{3.31}) and
(\ref{3.33}), then we get
\begin{align*}
A_1 \geq&\frac{\kappa_1^{m-1} Q}{P_m
S_{\mu}^2}[(1+\frac{\beta}{2}-2\epsilon)\sum_{a\leq \mu}(
S_{\mu}^{aa}h_{aa1})^2-C_{\epsilon}\sum_{a>\mu}(
S_{\mu}^{aa}h_{aa1})^2].
\end{align*}
For $a>\mu$, we have $$ S_{\mu}^{aa}\leq C\kappa_1\cdots\kappa_{\mu-1},\
\ \text{ and }  S_{\mu}\geq \kappa_1\cdots\kappa_{\mu}.$$ For $a\leq
\mu$, we have $$ S_{\mu}(\kappa|a)\leq
C(\frac{\kappa_1\cdots\kappa_{\mu}}{\kappa_a}+\frac{\kappa_1\cdots\kappa_{\mu+1}}{\kappa_{a}}).$$
Then, for sufficient large $\kappa_1$, we have
\begin{align}\label{3.34}
&P_m^2A_1\\
\geq&\kappa^{2m-2}_1 Q^{11}(1+\frac{\beta}{2}-2\epsilon)(1+\delta^m)\sum_{a\leq \mu}(1-\frac{C_3(1+\kappa_{\mu+1})}{\kappa_a})^2h_{aa1}^2-\frac{P_m\kappa_1^{m-3}\kappa_{\mu}^2 C_{\epsilon}}{\delta^2 S_{\mu}^2}\sum_{a>\mu}( S_{\mu}^{aa}h_{aa1})^2\nonumber\\
\geq&\kappa^{2m-2}_1
Q^{11}(1+\frac{\beta}{2}-2\epsilon)(1+\delta^m)(1-\frac{C_3(1+\delta'\kappa_{1})}{\delta\kappa_1})^2
\sum_{a\leq \mu}h_{aa1}^2-\frac{P_m\kappa_1^{m-3} C_{\epsilon}}{\delta^2}\sum_{a>\mu}h_{aa1}^2\nonumber\\
\geq&\kappa_1^{2m-2} Q^{11}\sum_{a\leq
\mu}h_{aa1}^2-\frac{P_m\kappa_1^{m-3}
C_{\epsilon}}{\delta^2}\sum_{a>\mu}h_{aa1}^2.\nonumber
\end{align}
Here, the last inequality comes from that we choose $\delta'$ and
$\epsilon$ satisfying
\begin{align*}
\delta'C_3\leq \epsilon \delta, \frac{C_3}{\kappa_1}\leq
\epsilon\delta \text{ and }
(1+\frac{\beta}{2}-2\epsilon)^2(1+\delta^m)\geq 1.
\end{align*}
Using  (\ref{3.19}) and (\ref{3.22}) or (\ref{3.34}),  we have
\begin{align}\label{3.35}
&P_m^2(A_1+B_1+C_1+D_1-E_1)\\
\geq& 2P_m\kappa_1^{m-2}\sum_{j\neq 1} Q^{jj}h_{jj1}^2
-\dfrac{C_{\epsilon}P_m\kappa_1^{m-3}}{\delta^2}\sum_{j>\mu}h^2_{jj1}.\nonumber
\end{align}
Now, for $k\geq j> \mu$,  we have
\begin{align}
\kappa_1 Q^{jj}= &\kappa_1\sum_{s>j}\alpha_s\sigma_s^{jj}+\kappa_1\sum_{\mu<s\leq j}\alpha_s\sigma_s^{jj}\\
\geq&\sum_{s>j}\frac{\alpha_s\kappa_1\cdots\kappa_{s}\cdot\kappa_1}{\kappa_j}+\sum_{\mu<s\leq j} \frac{\alpha_s\kappa_1\cdots\kappa_{s}\cdot\kappa_1}{\kappa_s}\nonumber\\
\geq&\frac{\kappa_1}{C_4\kappa_j}\sum_{s>j}\alpha_sQ_s+\frac{\kappa_1}{C_4\kappa_s}\sum_{\mu<s\leq
j}\alpha_sQ_s\geq \frac{Q}{C_4\delta'}.\nonumber
\end{align}
For $j\geq k+1$, similar to the above argument, we have
\begin{align}
\kappa_1
Q^{jj}\geq\frac{\kappa_1}{C_4}\sum_{s>\mu}\frac{\alpha_sQ_s}{\kappa_s}\geq
\frac{ Q}{C_4\delta'}.\nonumber
\end{align}
For both cases, chose $\delta'$ small enough satisfying
$$\delta'<\dfrac{ Q\delta^2}{C_4C_{\epsilon}},$$ then
(\ref{3.35}) is nonnegative.  We have our desired results.
\end{proof}

Hence, by Lemma \ref{lemma8} and Lemma
\ref{lemma2} and similar argument in \cite{GRW, LRW}, we have the following Corollary.
\begin{coro}\label{cor4}
There exists some finite sequence of positive numbers
$\{\delta_i\}_{i=1}^k$, such that, if the following inequality holds
for some index $1\leq r \leq k-1$, $$\frac{\kappa_r}{\kappa_1} \ \
\geq \ \ \delta_r, \text{ and } \frac{\kappa_{r+1}}{\kappa_1}\leq
\delta_{r+1},$$ and $\alpha_1=\alpha_2=\cdots=\alpha_{r}=0$ in the
expression \eqref{EQ}, then, for sufficiently large $K$, we have
\begin{align}\label{3.36}
 A_1+B_1+C_1+D_1- E_1 \geq 0.
\end{align}\end{coro}

Now, we can prove Theorem \ref{theo2}.

By Corollary \ref{cor4}, there exists some sequence $\{\delta_i\}_{i=1}^k$. We use similar tricks as in \cite{GRW, LRW}. The only difference is that for case (B), if there is some index $1\leq r\leq k-1$ satisfying $$\kappa_r\geq \delta_{r}\kappa_1, \ \ \text{ and }
\kappa_{r+1}\leq \delta_{r+1}\kappa_1,$$ and some another index $s\leq r$, such that $\alpha_s\neq 0$, then, by equation, we have
$$\alpha_s\sigma_s(\kappa_1,\cdots,\kappa_n)\leq Q(\kappa_1,\cdots,\kappa_n)=\psi(X,\nu),$$ which implies the bound of $\kappa_1$.

\section{The sum type }
 In this section, we will discuss the simplest type of our equations \eqref{EQ}.
It is the sum type equations defined by \eqref{1.g1}, where integer $k\geq 1$ and $\alpha\geq 0$.

Corresponding our problem, we need a new convex cone $\tilde\Gamma_k$ compatible our equations. We define
\begin{equation}\label{def G1}
\tilde\Gamma_k=\Gamma_{k-1}\ \cap \{ \la \in \mathbb{R}^n;
\al\sigma_{k-1}(\la)+\sigma_{k}(\la)>0\}.\end{equation}
Here $\Gamma_k$ is
the $k$-convex Garding's cone introduced by Caffarelli-Nirenberg-Spruck
\cite{CNS3},
\begin{defi}\label{k-convex} For a domain $\Omega\subset \mathbb R^n$, a function $v\in C^2(\Omega)$ is called $k$-convex if the eigenvalues $\la (x)=(\la_1(x), \cdots, \la_n(x))$ of the hessian $\nabla^2 v(x)$ is in $\Gamma_k$ for all $x\in \Omega$, where $\Gamma_k$ is the Garding's cone
\begin{equation}\label{def G2}
\Gamma_k=\{\la \in \mathbb R^n \ | \quad \sigma_m(\la)>0,
\quad  m=1,\cdots,k\}.
\end{equation}
\par
A $C^2$ regular hypersurface $M\subset R^{n+1}$ is $k-$convex if its principal curvature
$\kappa(X)\in \Gamma_k$ for all $X\in M$.
\end{defi}
 By the definition of $\Gamma_k$ and $\tilde\Gamma_k$, we note that
$$
\Gamma_{k-1}\supseteq \tilde\Gamma_k\supseteq \Gamma_k.
$$

 This section is composed by two parts.  In the first part, we will prove that  the cone $\tilde{\Gamma}_k$ is  the suitable
admissible solutions set. Then in second part, we will discuss the quotient concavity and it can imply the curvature estimates for sum type equations.

\begin{theo}\label{theo5}
The set $\tilde\Gamma_k$ is a convex set. In the cone $\tilde\Gamma_k$, the equation
(\ref{1.g1}) is elliptic.
\end{theo}
\begin{proof}
At first, let's prove the convexity of  $\tilde\Gamma_k$. Suppose $\la,\tilde{\la}\in\tilde\Gamma_k$ and  $0\leq t\leq 1$. We let
$$\la_t=t\la+(1-t)\tilde{\la}.$$ Then we have
\begin{eqnarray}\label{4.4}
&&Q_S^k(\la_t)\\
&=&\sigma_{k}(\la_t)+\al\sigma_{k-1}(\la_t)\nonumber\\
&=&t^{k}\sigma_{k}(\la)+(1-t)^{k}\sigma_{k}(\tilde\la)
+\sum_{l=1}^{k-1}a_lt^l(1-t)^{k-l}\sigma_{l,k-l}(\la,\tilde\la)\nonumber\\
&&+\al t^{k-1}\sigma_{k-1}(\la)+\al(1-t)^{k-1}\sigma_{k-1}(\tilde\la)
+\al\dsum_{s=1}^{k-2}b_st^s(1-t)^{k-1-s}\sigma_{s,k-1-s}(\la, \tilde\la)\nonumber\\
&\geq& t^{k-1}(\al
\sigma_{k-1}(\la)+t\sigma_{k}(\la))+(1-t)^{k-1}(\al\sigma_{k-1}(\tilde\la)+(1-t)\sigma_{k}(\tilde\la)).\nonumber
\end{eqnarray}
Here $a_l,b_l$ are two sequences of positive constants and $\sigma_{l,k-l},\sigma_{s,k-1-s}$ are the polarization of $\sigma_k$ and $\sigma_{k-1}$. Let's explain the last inequality. Since $\lambda,\tilde\la$ both are
in $\Gamma_{k-1}$,  by the definition of Garding's cone and
 Garding's inequality, we have
$$\sigma_{l,k-l}(\la,\tilde\la)\geq \sigma_l(\la)\sigma_{k-l}(\tilde{\la})>0,\sigma_{s,k-1-s}(\la,\tilde\la)\geq \sigma_s(\la)\sigma_{k-1-s}(\tilde\la)>0,$$ for $1\leq l\leq k-1$ and $1\leq s\leq k-2$.

Now, we can prove that $Q_S^k(\la_t)$ is nonnegative. In fact, we clearly have $\sigma_{k-1}(\la)>0, \sigma_{k-1}(\tilde\la)>0$. If $\sigma_k(\la)>0$, the first term of the last line in \eqref{4.4} is nonnegative. If $\sigma_k(\la)\leq 0$, we have
$$\al\sigma_{k-1}(\la)+t\sigma_{k}(\la)\geq \al\sigma_{k-1}(\la)+\sigma_{k}(\la)>0,$$ for $0\leq t\leq 1$. The second term of the last line in \eqref{4.4} can be discussed in the same way which implies $Q_S^k(\la_t)$ is nonnegative. Thus, in any case, we have the convexity.
\par
Then, we will prove the elliptic. P. Guan and C.-S. Lin \cite{LG} observed that
function $\dfrac{\sigma_{k}}{\sigma_{k-1}}$ is a degenerated elliptic function in
$\Gamma_{k-1}$ cone and it is only degenerate on the set $\{\sigma_k=0\}$(i.e.
it is strongly elliptic even in the set $\{\sigma_{k}<0\}\cap\Gamma_{k-1}$).
In fact, the degenerated elliptic property can be easily get by the following calculation. Since we have
\begin{align*}
\abc{\dfrac{\sigma_{k}}{\sigma_{k-1}}}^{ii}
&=\dfrac{\sigma_k^{ii}\sigma_{k-1}-\sigma_k\sigma_{k-1}^{ii}}{\sigma_{k-1}^2}\\
&=\dfrac{\sigma_{k-1}(\lambda|i)\sigma_{k-1}(\lambda|i)-\sigma_k(\lambda|i)\sigma_{k-2}(\lambda|i)}{\sigma_{k-1}^2},
\end{align*}
then, by Newton inequality, it is non negative. Furthermore, if
$\sigma_{k-1}(\lambda|i)\neq 0$, it is positive.
\par
Hence, for our equation \eqref{4.1}, direct calculation shows
\begin{align*}
(Q_S^k)^{ii}=&\sigma_{k}^{ii}+\al\sigma_{k-1}^{ii}=\sigma_{k-1}\abc{\dfrac{\sigma_{k}^{ii}}{\sigma_{k-1}}+\al
\dfrac{\sigma_{k-1}^{ii}}{\sigma_{k-1}}}\\
=&\sigma_{k-1}\abc{\abc{\dfrac{\sigma_{k}}{\sigma_{k-1}}}^{ii}+\dfrac{\sigma_{k-1}^{ii}}{\sigma_{k-1}}\abc{\al
+\dfrac{\sigma_{k}}{\sigma_{k-1}}}}\\
=&\sigma_{k-1}\abc{\abc{\dfrac{\sigma_{k}}{\sigma_{k-1}}}^{ii}+\dfrac{\sigma_{k-1}^{ii}}{\sigma_{k-1}}\dfrac{\psi}{\sigma_{k-1}}}>0,
\end{align*}
which gives  the elliptic in $\tilde\Gamma_k$.

\par

\end{proof}
\par
Before we discuss the quotient concavity, we need the following little Lemma.

\begin{lemm}\label{lem13}
$\dfrac{\sigma_k(\la)}{Q_S^k(\la)}$ is concave for any  $\la\in
\tilde\Gamma_k$.
\end{lemm}
\begin{proof}
It is well know that function $\dfrac{\sigma_k(\la)}{\sigma_{k-1}(\la)}$
is concave for  $\la\in\Gamma_{k-1}\supseteq \tilde\Gamma_k $, which implies that
$$\al+\dfrac{\sigma_k(\la)}{\sigma_{k-1}(\la)}=\dfrac{\al\sigma_{k-1}(\la)+\sigma_k(\la)}{\sigma_{k-1}(\la)}$$
is also a concave function.  Thus, we obtain that function
$$\dfrac{\sigma_k(\la)}{Q_S^k(\la)}=1-\dfrac{\al\sigma_{k-1}(\la)}{\al\sigma_{k-1}(\la)+\sigma_k(\la)}$$
is  concave for $\la\in \tilde{\Gamma}_k$ and positive $\alpha$.
\end{proof}
Then, we have the following quotient concavity similar to Hessian equations. Our idea comes from \cite{HS, ML}.
\begin{lemm}\label{lem14}
\par
The quotient functions
$$q_k(\la)=\dfrac{Q_S^{k+1}(\la)}{Q_S^{k}(\la)}$$
are concave functions for $\la\in \Gamma_{k}$.
\end{lemm}
\begin{proof}
We use induction to prove it. Let's consider $q_1$ at first. For all $\la\in\Gamma_1$ and $\xi\in R^n$, we have
\begin{eqnarray}
&&2q_1(\la)-q_1(\la+\xi)-q_1(\la-\xi)\label{e2} \\
&=&\frac{2\al^2\sigma_1^2(\xi)+2\sigma_1(\xi)[\al+\sigma_1(\la)]\sum_{p\neq
q}\la_p\xi_q-2\sigma_2(\xi)[\al+\sigma_1(\la)]^2-2\sigma_2(\la)\sigma_1^2(\xi)}{[\al+\sigma_1(\la)][\al+\sigma_1(\la+\xi)][\al+\sigma_1(\la-\xi)]}.\nonumber
\end{eqnarray}
Note that
\begin{eqnarray}
&&2\sigma_1(\xi)[\al+\sigma_1(\la)]\sum_{p\neq q}\la_p\xi_q \label{e3}\\
&=&2\sigma_1(\xi)[\al+\sigma_1(\la)]\sum_{p}\la_p(\sigma_1(\xi)-\xi_p) \nonumber\\
&=&\sigma_1^2(\xi)[\al+\sigma_1(\la)]^2+\sigma_1^2(\xi)[\sigma_1^2(\la)-\al^2]-2\sigma_1(\xi)[\al+\sigma_1(\la)]\sum_{p}\la_p\xi_p.
\nonumber
\end{eqnarray}
Combining (\ref{e3}) and (\ref{e2}), we get
\begin{eqnarray}
2q_1(\la)-q_1(\la+\xi)-q_1(\la-\xi)
=\frac{\al^2\sigma_1^2(\xi)+\dsum_i[(\al+\sigma_1(\la))\xi_i-\sigma_1(\xi)\la_i]^2}{[\al+\sigma_1(\la)][\al+\sigma_1(\la+\xi)][\al+\sigma_1(\la-\xi)]}\nonumber
\end{eqnarray}
Using the above equality, for all $\la\in\Gamma_1$ and $\xi\in R^n$, we obtain
\begin{align*}
-\dfrac{\pa^2 q_1}{\pa \xi^2}(\la)&=\dlim_{\varepsilon\ra
0}\dfrac{2q_1(\la)-q_1(\la+\varepsilon\xi)-q_1(\la-\varepsilon\xi)}{\varepsilon^2}\\
&=\dfrac{\al^2\sigma_1^2(\xi)+\dsum_i[(\al+\sigma_1(\la))\xi_i-\sigma_1(\xi)\la_i]^2}{[\al+\sigma_1(\la)]^3}.
\end{align*}
Thus, the right-hand side of the last line is positive which implies that $q_1$ is a concave function.
\par
For arbitrary $k$, if we assume that
$q_{k-1}$ is concave, we will prove that $q_{k}$ is also concave.
Using
$$\sum_i\la_i\sigma_{k}(\la |i)=(k+1)\sigma_{k+1}\text{ and }
\sigma_{k}(\la |i)=\sigma_{k}-\la_i\sigma_{k-1}(\la  | i),$$ we have
\begin{align}\label{e5}
(k+1)q_{k}(\la)-\frac{\al\sigma_{k}(\la)}{Q_S^k(\la)}&=\sum_i\la_i\frac{Q_S^k(\la |i)}{Q_S^k(\la)}\\
&=\sum_i\la_i\frac{Q_S^{k}(\la)-\la_iQ_S^{k-1}(\la |i)}{Q_S^k(\la)}\nonumber\\
&=\sum_i[\la_i-\frac{\la_i^2}{\la_i+q_{k-1;i}(\la)}].\nonumber
\end{align}
Here the notation $q_{k-1:i}(\la)$ means $q_{k-1}(\la |i)$, namely, excluded $i$ in the indices set $1,\cdots, n$. Take sufficiently small $\xi\in R^n$ satisfying  $\la\pm\xi\in\Gamma_k$. By
(\ref{e5}), we get
\begin{align}
&(k+1)[2q_k(\la)-q_k(\la+\xi)-q_k(\la-\xi)]-\alpha[\frac{2\sigma_k}{Q_S^k}(\la)-\frac{\sigma_k}{Q_S^k}(\la+\xi)-\frac{\sigma_k}{Q_S^k}(\la-\xi)]\nonumber\\
=&\sum_i\left(\frac{(\la_i+\xi_i)^2}{\la_i+\xi_i+q_{k-1;i}(\la+\xi)}+\frac{(\la_i-\xi_i)^2}{\la_i-\xi_i+q_{k-1;i}(\la-\xi)}\right.\nonumber\\
&\left.-\frac{(2\la_i)^2}{2\la_i+q_{k-1;i}(\la+\xi)+q_{k-1;i}(\la-\xi)}\right)\nonumber\\
&+\sum_i\abc{\frac{(2\la_i)^2}{2\la_i+q_{k-1;i}(\la+\xi)+q_{k-1;i}(\la-\xi)}-\dfrac{2\la_i^2}{\la_i+q_{k-1;i}(\la)}}\nonumber\\
=&\sum_i\frac{[(\la_i+\xi_i)q_{k-1;i}(\la-\xi)-(\la_i-\xi_i)q_{k-1;i}(\la+\xi)]^2}{[\la_i+\xi_i+q_{k-1;i}(\la+\xi)][\la_i-\xi_i+q_{k-1;i}(\la-\xi)][2\la_i+q_{k-1;i}(\la+\xi)+q_{k-1;i}(\la-\xi)]}\nonumber\\
&+2\sum_i\la_i^2\dfrac{2q_{k-1;i}(\la)-q_{k-1;i}(\la+\xi)-q_{k-1;i}(\la-\xi)}{[\la_i+q_{k-1;i}(\la)][2\la_i+q_{k-1;i}(\la+\xi)+q_{k-1;i}(\la-\xi)]}.\nonumber
\end{align}
Using Lemma \ref{lem13}, we have
$$
\dfrac{2\sigma_k}{Q_S^k}(\la)-\dfrac{\sigma_k}{Q_S^k}(\la+\xi)-\dfrac{\sigma_k}{Q_S^k}(\la-\xi)\geq
0.
$$
Thus, combining the previous two formulas, we obtain
\begin{align*}
&(k+1)[2q_k(\la)-q_k(\la+\xi)-q_k(\la-\xi)]\\
\geq&2\dsum_i\la_i^2\dfrac{2q_{k-1;i}(\la)-q_{k-1;i}(\la+\xi)-q_{k-1;i}(\la-\xi)}{[\la_i+q_{k-1;i}(\la)][2\la_i+q_{k-1;i}(\la+\xi)+q_{k-1;i}(\la-\xi)]},
\end{align*}
which implies that, for any $\xi$, we have
\begin{align*}
-\dfrac{\pa^2q_k}{\pa\xi^2}(\la)=&\dlim_{\varepsilon\ra
0}\dfrac{2q_k(\la)-q_k(\la+\varepsilon\xi)-q_k(\la-\varepsilon\xi)}{\varepsilon^2}\\
\geq&\dlim_{\varepsilon\ra
0}\dsum_i\dfrac{2\la_i^2}{k+1}\dfrac{2q_{k-1;i}(\la)-q_{k-1;i}(\la+\varepsilon\xi)-q_{k-1;i}(\la-\varepsilon\xi)}{\varepsilon^2[\la_i+q_{k-1;i}(\la)][2\la_i+q_{k-1;i}(\la+\varepsilon\xi)+q_{k-1;i}(\la-\varepsilon\xi)]}\\
=&-\dsum_i\dfrac{\la_i^2}{k+1}\dfrac{\dfrac{\pa^2q_{k-1;i}}{\pa\xi^2}(\la)}{[\la_i+q_{k-1;i}(\la)]^2}\\
=&-\dsum_i\dfrac{}{}\dfrac{\la_i^2\dfrac{\pa^2q_{k-1}}{\pa[\xi]_i^2}([\la]_i)}{(k+1)[\la_i+q_{k-1;i}(\la)]^2}\geq
0.
\end{align*}
Here $[\xi]_i,[\la]_i$ denote the vectors $\xi$ and $\la$ of which
the $i$-th component is vanish. We obtain the concavity for $q_k$.
\end{proof}
A direct corollary of the above theorem is that
\begin{coro}\label{coro15}
For any $1\leq l<k\leq n$, the two functions
$$\left(\frac{Q_S^k}{Q_S^l}(\la)\right)^{\frac{1}{k-l}}, \text{ and } (Q_S^k(\la))^{\frac{1}{k}}$$
are concave functions in the admissable cone $\tilde\Gamma_k$.
\end{coro}
\begin{proof}
By the definition of $q_k$ in the previous Lemma, we observe that
$$\left(\frac{Q_S^k}{Q_S^l}\right)^{\frac{1}{k-l}}=(q_{k-1}q_{k-2}\cdots q_l)^{\frac{1}{k-l}},$$ and
$$(Q_S^k)^{\frac{1}{k}}=(q_{k-1}q_{k-2}\cdots q_1 Q_S^1)^{\frac{1}{k}}.$$
Since, again by the previous Lemma, $q_{k-1},\cdots, q_1$ and $Q_S^1$ are all concave functions,
our corollary comes from some well known fact that the geometric mean of the finite positive concave functions is also a concave function.
\end{proof}
The above quotient concavity directly leads to Theorem \ref{theosum}.

\section{Discussion of quotient concavity }
In section 2, we see that Condition (Q) is the key for the curvature estimates. The previous section gives the quotient concavity for sum type. In this section, we try to study the quotient concavity for the general type, namely, linear combination of $\sigma_m$.

For any $x\in \mathbb{R}^n$, we let $\bar\alpha_{k-m}=\alpha_m$ in \eqref{QSS}.
Without loss of generality, we assume that $\alpha_k=\bar{\alpha}_0=1$.
Let $\theta=(1,1,1,\cdots,1)$. For any $t\in \mathbb{R}$ we have
$$\sigma_m(t\theta+x)=\frac{(n-k)!}{(n-m)!}\frac{d^{k-m}}{d t^{k-m}}\sigma_k(t\theta+x).$$ Therefore, by \eqref{QSS}, we have
$$Q(t\theta+x)=\sum_{m=0}^k\frac{(n-k)!\bar\alpha_{m}}{(n-k+m)!}\frac{d^m}{dt^m}\sigma_k(t\theta+x).$$
Now by Condition (C), for $\bar\alpha_m$, there is some  $b\in \mathbb{R}^N$ and $N\geq k$ such that for $m=0,1,\cdots, n$, we have
\begin{eqnarray}\label{C}
\alpha_m'=\frac{(n-k)!\bar\alpha_{m}}{(n-k+m)!}=\sigma_m(b).
\end{eqnarray}
 With our condition, we can rewrite that
$$Q(t\theta+x)=\prod_{m=1}^N(1+b_m\frac{d}{dt})\sigma_k(t\theta+x),$$
where $b_1,\cdots, b_N$ are the component of vector $b$.\begin{lemm}\label{lem} For any $x\in \mathbb{R}^n$, the order $k$ polynomial $Q(t\theta+x)$ always has $k$ real roots.
\end{lemm}
\begin{proof}
By induction, we only need to prove that $$(1+b_m\frac{d}{dt})\sigma_k(t\theta+x)$$ has $k$ real roots. For $b_m\neq 0$,  this comes from
$$(1+b_m\frac{d}{dt})\sigma_k(t\theta+x)=b_me^{-\frac{t}{b_m}}\frac{d}{dt} (e^{\frac{t}{b_m}}\sigma_{k}(t\theta+x)),$$ and generalized Roll's theorem. In fact, for $b_m>(<)0$, $e^{\frac{t}{b_m }}\sigma_{k}(t\theta+x)$
has $k+1$ roots, since, $\sigma_k(t\theta+x)$ has $k$ roots and $-\infty(+\infty)$ is also a root.
For $b_m=0$, it is obvious.
\end{proof}

Hence we have the following Proposition.
\begin{prop}\label{prop2}
For $a\in \mathbb{R}^n$ and $\sigma_k(a)\neq 0$,  for any $x\in \mathbb{R}^n$, we can write that
$$Q(at+x)=\sigma_k(a)\prod_{m=1}^k(t+\lambda_m(x,a)).$$ Namely, $Q(at+x)$ has $k$ real roots $-\la_1(x,a),\cdots,-\la_n(x,a)$.
\end{prop}
\begin{proof}
The proof is following the paper \cite{HL}. We consider the following polynomial, for $s\in \mathbb{R}$,
$Q(s\theta+at+x)$. By Lemma \ref{lem}, we have $$Q(s\theta+at+x)=\prod_{m=1}^k(s+\lambda_m(at+x)).$$
Here the $\la_m(at+x)$ in fact depends on $\bar\alpha_0,\cdots,\bar\alpha_{k}$, but we will always use notion $\la_m(at+x)$ to present $\lambda_m(at+x,\bar\alpha_1,\cdots,\bar\alpha_k)$ without ambiguity.
We call the sequence  $\lambda_1(at+x), \lambda_2(at+x), \cdots, \lambda_k(at+x)$ the sign ordered eigenvalues, if  for $t>0$, we have $\la_1\leq\la_2\leq \cdots\leq \la_k$  and for $t<0$, we have $\la_1\geq\la_2\geq \cdots\geq \la_k$. We call the sequence $\la_1,\la_2\cdots,\la_k$ the ordered eigenvalues if we require $\la_1\leq \la_2\leq\cdots\leq \la_k$. Obviously, the sign ordered eigenvalues are ordered eigenvalues for $t>0$.

We know that $\lambda_m$ is continuous. By the same argument
as Lemma 2.10 in \cite{HL}, we also can prove that $\lambda_m(at+x)$ is real analytic for variable $t$. Now, let's follow the argument of Theorem 2.9 in \cite{HL}.  For $t\neq 0$, we have
\begin{eqnarray}\label{1}
Q(s\theta+at+x)&=&\sum_{m=0}^k\bar\alpha_{k-m}\sigma_m(s\theta+at+x)\\
&=&t^k\sum_{m=0}^k\frac{\bar\alpha_{k-m}}{t^{k-m}}\sigma_m(\frac{s}{t}\theta+a+\frac{x}{t}).\nonumber
\end{eqnarray}
It is obvious that
$$\frac{\bar\alpha_{k-m}}{t^{k-m}}=\frac{(n-m)!}{(n-k)!}\sigma_{k-m}(\frac{b}{t}).$$
Hence, by Lemma \ref{lem} and \eqref{1}, we have
$$Q(s\theta+at+x)=t^k\prod_{m=1}^k(\frac{s}{t}+\lambda_m(a+\frac{x}{t}, \frac{\bar\alpha_1}{t}, \frac{\bar\alpha_2}{t^2},\cdots,\frac{\bar\alpha_k}{t^k})).$$
Here $\la_1\leq \la_2\leq\cdots\leq \la_k$ are ordered eigenvalues.
Hence, we obtain, for $m=1,\cdots,n$,
\begin{eqnarray}\label{2}
\lambda_m(at+x,\bar\alpha_1,\bar\alpha_2,\cdots,\bar\alpha_k)=t\lambda_m(a+\frac{x}{t}, \frac{\bar\alpha_1}{t}, \frac{\bar\alpha_2}{t^2},\cdots,\frac{\bar\alpha_k}{t^k})).
\end{eqnarray}
By the continuity of $\lambda_m$, for $t\rightarrow \pm\infty$, we have
$$\lambda_m(a+\frac{x}{t}, \frac{\bar\alpha_1}{t}, \frac{\bar\alpha_2}{t^2},\cdots,\frac{\bar\alpha_k}{t^k})\rightarrow \lambda_m(a,0,0,\cdots, 0).$$
It is clear that
$$\sigma_k(a)=\prod_{m=1}^k\lambda_m(a,0,0,\cdots,0)\neq 0.$$
Since  $ \lambda_m(a, 0,\cdots,$ $0)\neq 0$, then, for $t\rightarrow \pm\infty$, we have  $$\lambda_m(at+x)\rightarrow \pm(\mp)\infty.$$ By continuity, $\lambda_m(at+x)$   always has one real solution $t_m$, namely, $\lambda_m(at_m+x)=0$. Define some set $S=\{t_1,t_2,\cdots,t_k\}$.
We claim that if $t_{m_1}=t_{m_2}=\cdots=t_{m_l}=\tau$ in $S,$ then$$\text{  $\tau$ is  at least $l$ multiple root of the polynomial $Q(at+x)$}.$$
It is obvious that $$Q(at+x)=\prod_{m=1}^k\lambda_m(at+x).$$ We know that $Q(a\tau+x)=0$.
Since $\lambda_m(at+x)$ is real analytic for $t$, then
$$\frac{d^{\alpha}Q}{dt^{\alpha}}(a\tau+x)=0,$$ for $0\leq \alpha\leq l-1$. The claim is proved. By the claim, we know that $t_1,\cdots, t_k$ are exactly $k$ real roots for polynomial $Q(at+x)$.
\end{proof}

Now we can obtain some concavity.
\begin{lemm}\label{lemma} If for any $x,y\in\mathbb{R}^n$ and $\sigma_k(x)\neq 0$, $Q(xt+y)$ has $k$ real roots, then,
$(Q(x))^{1/k}$ is a concave function in $x\in \Gamma_k$. Thus, if condition (C) holds, for $Q(x)$ defined by \eqref{QSS}, then $(Q(x))^{1/k}$ is a concave function in $\Gamma_k$.
\end{lemm}
\begin{proof}For any $x,y\in \Gamma_k$ and $0\leq t\leq 1$, we need to prove that
$Q(yt+(1-t)x)^{1/k}$ is a concave function. Let $a=y-x$, and then we have $$Q(yt+(1-t)x)^{1/k}=Q(at+x)^{1/k}.$$ If $\sigma_k(a)\neq 0$, by Lemma \ref{prop2}, we have
$$Q(at+x)=\sigma_k(a)\prod_{m=1}^k(t+\lambda_m(x,a))=\sigma_k(a)\sigma_k(t\theta+\lambda(x,a)),$$ where $\lambda(x,\theta)=(\lambda_1(x,\theta),\cdots,\lambda_k(x,\theta))$. Then, by Newton-Maclaurin inequality, we get
\begin{eqnarray}
&&\frac{d^2}{dt^2}Q(at+x)^{1/k}\nonumber\\
&=&\frac{\sigma^2_k(a)}{n}Q(at+x)^{\frac{1}{k}-2}(2\sigma_k(t\theta+\lambda)\sigma_{k-2}(t\theta+\lambda)-\frac{k-1}{k}\sigma_{k-1}^2(t\theta+\lambda))\nonumber\\
&\leq& 0\nonumber.
\end{eqnarray}
For $\sigma_k(a)=0$, we can have some sequence $a_l\rightarrow a$ and $\sigma_k(a_l)\neq 0$. Then, the previous results tells us that
$$\frac{d^2}{dt^2}Q(a_lt+x)^{1/k}\leq 0.$$ Taking $l\rightarrow \infty$, we obtain our result.
\end{proof}
Now, we study the quotient concavity of  these polynomials.
For  $$Q(t\theta+x)=\prod_{m=1}^N(1+b_m\frac{d}{dt})\sigma_k(t\theta+x),$$ we denote that
\begin{eqnarray}\label{QN}
Q^{N'}_{k-1}(t\theta+x)=\prod_{m=1}^{N'}(1+b_m\frac{d}{dt})\sigma_{k-1}(t\theta+x),
\end{eqnarray}
which is a order $k-1$ polynomial.
\begin{lemm}\label{lem20}
For  $x\in \Gamma_k$,  functions $$\frac{Q(x)}{Q^N_{k-1}(x)}, \text{ and } \frac{Q(x)}{Q^{N-1}_{k-1}(x)}$$ are concave functions.
\end{lemm}
\begin{proof}
For $x,y\in \Gamma_k$, let $ty+(1-t)x=at+x$, where $a=y-x$. If $\sigma_k(a)\neq 0$, by Proposition \ref{prop2}, we have $$Q(s\theta+at+x)=\sigma_k(a)\prod_{m=1}^k(t+\lambda_m(s\theta+x,a)).$$ Since, it is clear that, by \eqref{QN}, we have
\begin{eqnarray}
\frac{Q(s\theta+at+x)}{Q^N_{k-1}(s\theta+at+x)}=(n-k+1)\frac{Q(s\theta+at+x)}{\dfrac{d}{ds}Q(s\theta+at+x)}\nonumber.
\end{eqnarray}
Hence, taking $s=0$, we get
\begin{eqnarray}
\frac{Q(at+x)}{Q^N_{k-1}(at+x)}=\frac{(n-k+1)\sigma_k(a)\prod_{m=1}^k(t+\lambda_m(x,a))}{\sigma_k(a)\sum_{l=1}^k\dfrac{d \lambda_l}{ds}(x,a)\prod_{m\neq l}(t+\lambda_m(x,a))}=\frac{n-k+1}{\sum_{l=1}^k\dfrac{d \lambda_l}{ds}(x,a)\dfrac{1}{t+\lambda_l(x,a)}}\nonumber.
\end{eqnarray}
By the same argument as Lemma 2.10 in \cite{HL}, we also can prove that $\lambda_m(s\theta+x,a)$ is real analytic for variable $s$. Then
$$\frac{d\lambda_l}{ds}(x,a)=\left.\frac{d}{ds}(\lambda_l(s\theta+x,a))\right |_{s=0} $$ is well defined.

We assume the ordered eigenvalues of $\sigma_k(a)$ are $\mu_1\leq \mu_2\leq \cdots\leq \mu_{K_0}< 0 < \mu_{K_0+1}\leq \cdots\leq \mu_k$.
For any fixed $0\leq t\leq 1$ and $\tau\in \mathbb{R}$, we have
\begin{eqnarray}
Q(s\theta+a\tau+at+x)&=&C_n^k\prod_{m=1}^k(s+\lambda_m(a\tau+at+x))\nonumber\\
&=&C_n^k\prod_{m=1}^k(s+\tau\lambda_m(a+\frac{at+x}{\tau}, \frac{\bar\alpha_1}{\tau}, \frac{\bar\alpha_2}{\tau^2},\cdots,\frac{\bar\alpha_k}{\tau^k}))\nonumber.
\end{eqnarray}
Assume that, for $\tau\geq0$,
\begin{eqnarray}\label{4}
\lambda_1(a\tau+at+x)\leq \lambda_2(a\tau+at+x)\leq\cdots\leq\lambda_k(a\tau+at+x)
\end{eqnarray}
are sign ordered eigenvalues.
Hence, we have, for any $\tau>0$,
$$\lambda_1(a+\frac{at+x}{\tau},\frac{\bar\alpha_1}{\tau},\cdots,\frac{\bar\alpha_k}{\tau^k})\leq \cdots\leq \lambda_k(a+\frac{at+x}{\tau},\frac{\bar\alpha_1}{\tau},\cdots,\frac{\bar\alpha_k}{\tau^k}).$$ If $\tau$ approaches to $+\infty$, we have
$$\lambda_1(a,0,0, \cdots,0)\leq \lambda_2(a,0,0,\cdots,0)\leq\cdots\leq \lambda_k(a,0,0\cdots,0).$$ That is the all $k$ ordered eigenvalues of $\sigma_k(a)$, which is the ordered eigenvalues.  Hence, we obtain that $\lambda_m(a,0,0,\cdots,0)$ $=\mu_m$.

By the continuity of the function $\lambda_m(a\tau+at+x)$, for any given $s_0$ and the equation $\la_m(a\tau+at+x)=s_0$ respect to variable $\tau$, we always have a unique solution. The argument is similar to Proposition \ref{prop2}.  Hence, we know that,
function $\lambda_m(a\tau+at+x)$ is a monotone function  for variable $\tau$. On the other hand, since $at+x\in \Gamma_k$ and $\alpha_m>0,$
it is clear that
$$\left.\frac{d^{l} Q(s\theta+at+x)}{ds^{l}}\right|_{s=0}>0,$$ for $0\leq l\leq k-1$, which implies
$$\sigma_l(\lambda(at+x))>0\ \ \text{ for } 0\leq l\leq k .$$ Here $\lambda(at+x)=(\lambda_1(at+x),\cdots\lambda_k(at+x))$. Thus, we obtain
$\lambda(at+x)\in \Gamma_k$, namely, $\lambda_m(at+x)>0$ for $1\leq m\leq k$. For $\mu_m>(<)0$, $\lambda_m(a\tau+at+x)$ is a monotone increasing(decreasing) function for $\tau$. Hence, the root of this function $\tau_m$ is negative (positive) since  $\lambda_m(at+x)>0$, which implies $\mu_m\lambda_m(at+x,a)>0$. Here $-\lambda_m(at+x,a)$ is the root of the polynomial $Q(a\tau+at+x)$, namely
$$Q(a\tau+at+x)=\sigma_k(a)\prod_{m=1}^k(\tau+\lambda_m(at+x,a)),$$ and $\tau_m+\la_m(at+x,a)=0$. Note that  $\la_m(at+x,a)$s are note sign ordered eigenvalues. Its index just comes from $\tau_m$.

Since, we also have
$$Q(a\tau+at+x)=\sigma_k(a)\prod_{m=1}^k(\tau+t+\lambda_m(x,a)),$$ then, combing the previous two equalities, we have
$$\lambda_m(at+x,a)=t+\lambda_m(x,a).$$ Thus, we get
$$\mu_m(t+\lambda_m(x,a))>0.$$

The discussion of $\tau_m$ gives us the order of $\tau_m$ which is $$\tau_{K_0+1}\leq\cdots\leq\tau_{k-1}\leq  \tau_k<0<\tau_1\leq \tau_2\leq\cdots\leq\tau_{K_0}.$$ Thus, we have the order of $\la_m(at+x,a)$,
\begin{eqnarray}
&\la_{K_0}(at+x,a)\leq\cdots \leq\la_2(at+x,a)\leq \la_1(at+x,a)\nonumber\\
&<0<\la_k(at+x,a)\leq \la_{k-1}(at+x,a)\leq\cdots\leq \la_{K_0+1}(at+x,a),\nonumber
\end{eqnarray}
which implies that
$$\la_{K_0}(x,a)\leq\cdots \leq\la_2(x,a)\leq \la_1(x,a)<0<\la_k(x,a)\leq \la_{k-1}(x,a)\leq\cdots\leq \la_{K_0+1}(x,a).$$
Since for $s\geq 0$, $s\theta+x\in \Gamma_k$, we also have
\begin{eqnarray}
&\la_{K_0}(s\theta+x,a)\leq\cdots \leq\la_2(s\theta+x,a)\leq \la_1(s\theta+x,a)\nonumber\\
&<0<\la_k(s\theta+x,a)\leq \la_{k-1}(s\theta+x,a)\leq\cdots\leq \la_{K_0+1}(s\theta+x,a).\nonumber
\end{eqnarray}
Hence, we have
\begin{eqnarray}
&\la_{K_0}(\theta,a,0,0,\cdots,0)\leq\cdots \leq \la_1(\theta,a,0,0\cdots,0)\nonumber\\
&<0<\la_k(\theta,a,0,0\cdots,0)\leq \cdots\leq \la_{K_0+1}(\theta,a,0,0,\cdots,0).\nonumber
\end{eqnarray}
Since $\lambda_m(s\theta+x,a)$ is also a monotone function, then the sign of the functions $$\frac{d \lambda_l}{ds}(x,a), \text{ and } \lambda_l(\theta, a,0,0,\cdots,0)$$ are same for any $l$.  Here, $\lambda_l(\theta,a,0,0,\cdots,0)$ are the roots of the polynomial
$$\sigma_k(at+\theta)=0.$$ The roots of the above polynomial in fact are
$$\frac{1}{\mu_{K_0}}\leq \frac{1}{\mu_{K_0-1}}\leq \cdots\leq \frac{1}{\mu_{1}}< 0 < \frac{1}{\mu_k}\leq\frac{1}{\mu_{k-1}}\leq\cdots\leq \frac{1}{\mu_{K_0+1}}.$$  Thus, we have
$$\la_l(\theta,a,0,0,\cdots,0)=\frac{1}{\mu_l},$$ which implies
$$\mu_l\frac{d \lambda_l}{ds}(x,a)>0.$$

Let's prove the concavity. Define some function $$g_l(t)=\frac{\mu_l(t+\lambda_l(x,a))}{\mu_l\frac{d \lambda_l}{ds}(x,a)}>0,$$ and $g_l''(t)=0$.
Then, we have
\begin{eqnarray}
\frac{1}{n-k+1}\frac{Q(at+x)}{Q^N_{k-1}(at+x)}=\frac{1}{\sum_m\dfrac{1}{g_m(t)}}.
\end{eqnarray}
The second order derivative of the above function is that
 $$-\frac{2}{(\sum_m\dfrac{1}{g_m(t)})^3}(\sum_m\dfrac{1}{g_m(t)}\sum_m\frac{(g_m'(t))^2}{g^3_m(t)}-(\sum_m\frac{g'_m(t)}{g^2_m(t)})^2)\leq 0.$$
For $\sigma_k(a)=0$, we take some $a_l$ with $\sigma_k(a_l)\neq 0$ and converges to $a$. Hence, the first function defined in our Lemma is a concave function. For the second function, we can rewrite it to be
$$\frac{Q(x)}{Q^{N-1}_{k-1}(x)}=\frac{Q^{N-1}_{k}(x)}{Q^{N-1}_{k-1}(x)}+(n-k+1)b_N.$$ Hence, it is also a concave function.
\end{proof}
A direct corollary of the above Lemma is the following result.
\begin{coro}\label{coro21} In $\Gamma_k$, the functions
$$(\frac{Q(x)}{Q^N_{l}(x)})^{1/(k-l)}  \text{ and }   (\frac{Q(x)}{Q^{N-l}_{k-l}(x)})^{1/l}$$ are concave functions. Especially, if $N=k$, functions $$(\frac{P_k(x)}{\sigma_1(x)+n\sum_mb_m})^{1/(k-1)} \text{ and } (\frac{P_k(x)}{\sigma_1(x)+nb_1})^{1/(k-1)}$$ are concave functions.
\end{coro}
\begin{proof}
Obviously, we have
$$\frac{Q(x)}{Q^N_{l}(x)}=\frac{Q(x)}{Q^N_{k-1}(x)}\frac{Q^N_{k-1}(x)}{Q^{N}_{k-2}(x)}\cdots\frac{Q^N_{l+1}(x)}{Q^N_{l}(x)}.$$ Hence, by Lemma \ref{lem20} and the proof of Corollary \ref{coro15}, we have our first result. For $N=k$, since, we have
$$Q_1^k(x)=\sigma_1(x)+n\sum_mb_m, Q_1^1=\sigma_1(x)+nb_1,$$ we have our second one using the concavity of the first two functions.
\end{proof}

We can use the previous result to revisit the quotient concavity of sum type equations.
\begin{coro} The function
$$q_k(x)=\frac{\sigma_{k+1}(x)+\al\sigma_k(x)}{\sigma_k(x)+\al\sigma_{k-1}(x)}$$ is a concave function in $\Gamma_{k+1}$ for $k\geq 1$.
\end{coro}
\begin{proof}Since, we have
$$\sigma_{k}(t\theta+x)=\frac{1}{n-k}\frac{d}{dt}\sigma_{k+1}(t\theta+x),\text{ and } \sigma_{k-1}(t\theta+x)=\frac{1}{n-k+1}\frac{d}{dt}\sigma_{k}(t\theta+x),$$ then, we get
\begin{eqnarray}
&&q_k(x)\nonumber\\
&=&\left.\frac{\sigma_{k+1}(t\theta+x)+\dfrac{\al}{n-k+1}\dfrac{d}{dt}\sigma_{k+1}(t\theta+x)}{\sigma_k(t\theta+x)+\dfrac{\al}{n-k+1}\dfrac{d}{dt}\sigma_{k}(t\theta+x)}\right|_{t=0}+\frac{\al}{n-k+1}\frac{\sigma_k(x)}{\sigma_k(x)+\al\sigma_{k-1}(x)}\nonumber.
\end{eqnarray}
Both of the above two functions are concave functions by Corollary \ref{coro21} and Lemma \ref{lem13}.  Hence, we have the concavity of $q_k(x)$ in $\Gamma_{k+1}$ cone.
\end{proof}

\begin{rema} The major difference between previous Corollary and Lemma \ref{lem14} is the definition field of the function $q_k$. Since, it is clear that $\Gamma_{k+1}\subset \tilde{\Gamma}_k$, Lemma \ref{lem14} is better.
\end{rema}

By Corollary \ref{coro21}, we can conclude the main result of this section by  Theorem \ref{theos4}.

\section{The conclusion }
Combing discussion of section 2 and section 4, we have our main result of this paper, Theorem \ref{main}. With appropriate barrier, we can prove the  existence result Theorem \ref{k-exist}   coming from the Theorem \ref{main}.  \\

\noindent {\bf Proof of Theorem \ref{k-exist}}:
The proof can be deduced by the degree theory as in \cite{gg, GRW}. We only give a brief review following \cite{GRW}. Consider the modified  auxiliary equation
\begin{eqnarray}\label{5.4}
\\
Q(\kappa(X))=\psi^t(X,\nu)=\big(tf^{-\frac1k}(X,\nu)+(1-t)(C^k_n[\frac{1}{|X|^k}+
\varepsilon(\frac{1}{|X|^k}-1)])^{-\frac1k}\big)^{-k}.\nonumber
\end{eqnarray}
The assumptions of $\psi^t$ satisfies the structural condition in the Constant Rank Theorem (Theorem 1.2 in \cite{GLM1}) which implies our convexity of solutions. Theorem \ref{main} gives curvature estimates. If we have $C^0$ bound, the proof of the rest part of this theorem is same as \cite{GRW}. The $C^0$ upper bound comes from our barrier condition with maximum principal. Same as \cite{GRW}, the lower bound only needs the uniform lower bound of the convex body's volume which we need to discuss here.  By our equation and uniform upper bound, there is some $m\leq k$ such that $\alpha_m\neq 0$ and
$$\sigma_m(\kappa)\leq C,$$ where $C$ is some constant only depending on $\psi$ and uniform upper bound. Thus, using Alexsandrov-Frenchel inequality and same argument in \cite{GRW}, we have the lower bound of the volume. \\

A corresponding $C^2$ estimates of convex solutions for  Hessian equations  defined in some domain also holds:
\begin{coro}
Suppose function $u$ defined in some domain $\Omega\subset \mathbb{R}^{n}$ is a
convex solution of the linear combination of $k$ Hessian equation
\begin{eqnarray}\label{EQ1}
Q(D^2 u)=\sum_{m=0}^k\alpha_m\sigma_m(D^2 u)=\psi(x,u,Du)
\end{eqnarray}
for some positive function $\psi(x,u,Du)\in
C^{2}(\mathbb{R}^n\times \mathbb{R}\times \mathbb{R}^n)$ and nonnegative coefficients $\alpha_0,\cdots,\alpha_k$ satisfying Condition (C), then there is a constant $C$ depending only on $n, k$, $\|u\|_{C^1}$,
$\inf \psi $ and $\|\psi\|_{C^2}$, such that
 \begin{equation}\label{5.6}
 \max_{\Omega} D^2 u \le C.\end{equation}
\end{coro}

In section 3, the admissible solution sets of the sum type equations have been obtained. Thus, we can state some existence results for admissible setting for sum type equations \eqref{1.g1}. We need pose some frequently used barrier \cite{BK, TW, CNS5}. We denote $\rho(X)=|X|$. Assume that

\noindent {\it Condition} (1). There are two positive constant $r_1<1<r_2$ such that
\begin{equation}\label{4.1}
\left\{
\begin{matrix}
\psi(X,\dfrac{X}{|X|}) &\geq&  \dfrac{Q_S^k(1,\cdots, 1)}{r^k_1},\ \ \text{ for } |X|=r_1,\\
\psi(X,\dfrac{X}{|X|}) &\leq&  \dfrac{Q_S^k(1,\cdots, 1)}{r_2^k}, \ \ \text{  for  } |X|=r_2 .
\end{matrix}\right.
\end{equation}

\noindent {\it Condition} (2). For any fixed unit vector $\nu$,
\begin{eqnarray}\label{4.2}
\frac{\p }{\p \rho}(\rho^k\psi(X,\nu))\leq 0,\ \   \text{ where } |X|=\rho.
\end{eqnarray}

For sum type equations, if the right hand side function $\psi$ only depends on position vector $X$, we have the admissible solution results.
\begin{theo}\label{exist} Suppose  that positive function $\psi\in C^2(\bar B_{r_2}\setminus B_{r_1})$ only depends position vector and satisfies conditions (\ref{4.1}) and (\ref{4.2}), then equation (\ref{1.g1}) has a unique $C^{3,\alpha}$ starshaped $\tilde{\Gamma}_k$ solution $M$ in $\{r_1\le |X|\le r_2\}$.
\end{theo}
\begin{proof}
The proof comes from Theorem \ref{theo5}, Corollary \ref{coro15} and standard argument for concave equations, seeing for example \cite{GJ} for more detail.
\end{proof}

For three special cases $k=1,2,3$ and $\al>0$, we have the following result.

\begin{theo}\label{exist2} Suppose $k=1,2,3$ and suppose positive function $\psi\in C^2(\bar B_{r_2}\setminus B_{r_1}\times \mathbb S^n)$ satisfies conditions (\ref{4.1}) and (\ref{4.2}), then equation (\ref{1.g1}) has a unique $C^{3,\alpha}$ starshaped $\tilde{\Gamma}_k$ solution $M$ in $\{r_1\le |X|\le r_2\}$.
\end{theo}
\begin{proof}
$k=1$ is the linear case which is well known. For $k=2,3$, in its admissible set, we have $$\sum_i(Q_S^k)^{ii}\geq C \sigma_1,$$ where
$C$ is some uniform constant. The curvature estimates for $2$-Hessian equations case at first obtained in \cite{GRW}. We also can generalize these idea to present cases. But we also can adopt the calculation used by Guan-Jiao \cite{GJ} and the previous formula to obtain curvature estimates for any concave functions. On the other hand, Corollary \ref{coro15} tell us that we can rewrite our equation to be some concave function. Thus we have our result.
\end{proof}

At last, inspired by \cite{RW}, the global $C^2$ estimates for $Q_S^{n}$
Hessian equation in $\tilde{\Gamma}_{n}$ also can be solved. That is the following existence theorem.
\begin{theo}
Suppose $k=n$ and suppose positive function $\psi\in C^2(\bar B_{r_2}\setminus B_{r_1}\times \mathbb S^n)$ satisfies conditions (\ref{4.1}) and (\ref{4.2}), then equation (\ref{1.g1}) has a unique $C^{3,\alpha}$ starshaped $\tilde{\Gamma}_k$ solution $M$ in $\{r_1\le |X|\le r_2\}$.
\end{theo}
The detail proof of this theorem will be carried out in a forthcoming paper by Ren \cite{CR}.   \\

\noindent {\it Acknowledgement:} The authors wish to thank  Professor Pengfei Guan to raise this problem and for his contribution with Professor C.-S. Lin about the admissible set. They also would like to thank Xiangwen Zhang for his interesting in their work and sharing unpublished note \cite{GZ}.
The work was started when the second and third author were visiting Shanghai Centre for Mathematical Sciences in 2014. They would like to  thank their various support and hospitality.  Part of the work was done while the first and last author were visiting McGill University. They also would  like to thank McGill University for their hospitality
and the support of CSC program during 2014-2015.

\end{document}